\documentclass[a4paper,12pt]{article}   
\usepackage[top=2.5cm,bottom=2.5cm,left=2.5cm,right=2.5cm]{geometry}
\usepackage{amsmath,amssymb,amsfonts,xspace}

\newtheorem{prethm}{{\bf Theorem}}

\newenvironment{theorem}{\begin{prethm}{\hspace{-0.5
               em}{\bf.}}}{\end{prethm}}

\newtheorem{prelem}{{\bf Lemma}}

\newenvironment{lem}{\begin{prelem}{\hspace{-0.5
               em}{\bf.}}}{\end{prelem}}  
               
\newtheorem{prelm}{{\bf Lemma}}

\newtheorem{prepro}{{\bf Proposition}}

\newtheorem{precor}{{\bf Corollary}}

\newenvironment{corollary}{\begin{precor}{\hspace{-0.5
               em}{\bf.}}}{\end{precor}}

\newtheorem{preobserv}{{\bf Observation}}

\newenvironment{observation}{\begin{preobserv}{\hspace{-0.5
               em}{\bf.}}}{\end{preobserv}}

\newtheorem{preconj}{{\bf Conjecture}}

\newtheorem{preremark}{{\bf Remark}}

\newtheorem{predef}{{\bf Definition}}

\newtheorem{preproof}{{\bf Proof.}}

\newenvironment{proof}[1]{\begin{preproof}{\rm
               #1}\hfill{$\Box$}}{\end{preproof}}

\newcommand{\sle}{{\rm SLEE}}

\renewcommand{\thefootnote}

\begin{document}
\vspace*{3cm}
\begin{center}
{\Large  Gutman index of the Mycielskian and its complement}

\vspace{5mm}
{\large  Mahdi Anbarloei$^*$, Ali Behtoei}

\vspace{4mm}
\baselineskip=0.20in
{\it Department of Mathematics, Imam Khomeini International University,} \\
{\it  P.O. Box: 34149-16818, Qazvin, Iran} \\
\footnote{* Corresponding author}
\footnote{e-mail: m.anbarloei@sci.ikiu.ac.ir,  a.behtoei@sci.ikiu.ac.ir}
\thispagestyle{empty}
\vspace{6mm}
\end{center}
\begin{abstract}
Let $G$ be a simple connected  graph. The Gutman index $Gut(G)$ of $G$
is defined as $\sum_{\{u,v\}\subseteq V(G)} d_G(u,v)\deg_G(u)\deg_G(v)$, where $\deg_G(u)$ is the degree of
vertex $u$ in $G$ and $d_G(u,v)$ is the distance between two vertices $u$ and $v$ in $G$.
In this paper, we study the Gutman index of Mycielskian of graphs.
Also, we determine exact value of the Gutman index of the complement of arbitrary Mycielskian graphs. 
\end{abstract}
\baselineskip=0.30in
\section{Introduction}

Throughout this paper we consider (non trivial) simple graphs, that are finite and undirected graphs without loops or multiple edges.
 Let $G=(V (G),E(G))$ be a connected graph of order $n=|V(G)|$ and of size $m=|E(G)|$.
The distance between two vertices $u$ and $v$ is denoted by  $d_G(u,v)$  which is the length of a shortest path between $u$ and $v$ in $G$.
The diameter of $G$ is $\max\{d_G(u,v):~u,v\in V(G)\}$. It is well known that almost all graphs have diameter two.
The degree of vertex $u$  is the number of edges adjacent to $u$ and is denoted by $\deg_G(u)$ .
A {\em chemical graph} is a graph whose vertices denote atoms and edges denote bonds between those atoms of any underlying chemical structure. 
A {\it topological index} for a (chemical) graph $G$ is a numerical quantity invariant under
automorphisms of $G$ and it does not depend on the labeling or pictorial
representation of the graph. 
Topological indices and graph invariants based on the distances between vertices of a graph or vertex degrees are widely used for
characterizing molecular graphs, establishing relationships between structure and properties of molecules, predicting
biological activity of chemical compounds, and making their chemical applications. 
The concept of topological index came from work done by Harold Wiener in 1947 while he was working on boiling point
of paraffin. 
 The {\em Wiener index} of $G$ is defined as
$W(G)=\sum_{\{u,v\}\subseteq V(G)} d_G(u,v)$, see \cite{JAMC2}.
Two important topological indices introduced about forty years ago 
by Ivan Gutman and Trinajsti$\acute{\mbox{c}}$  \cite{GutmanTrinajstic}  are the {\it first zagreb index} $M_1(G)$ and the {\it second zagreb index} $M_2(G)$ which are defined as below (see \cite{JAMC1}).
$$M_1(G)\!=\!\!\!\sum_{uv\in E(G)}\!\!(\deg_G(u)+\deg_G(v))\!=\!\!\!  \sum_{x\in V(G)}\!\!(\deg_G(x))^2,~M_2(G)\!=\!\!\!\sum_{uv\in E(G)}\!\!\!\deg_G(u)\deg_G(v).$$

The {\it degree distance} was introduced by Dobrynin and Kochetova \cite{Dobrynin} and Gutman \cite{Schultz-GutmanIndex} as a weighted version of the
Wiener index. The degree distance of G, denoted by DD(G), is defined as follows and it is computed for important families of graphs ( see\cite{MATCH-Degreedistance}  for instance):
$$DD(G)=\sum_{\{ u,v \} \subseteq V(G)} d_G(u,v)(\deg_G (u)+\deg_G(v)).$$

The {\em Gutman index} 
(another variant of the well known
and much studied Wiener Index) was introduced in 1994 by Gutman \cite{Schultz-GutmanIndex} as 
$$Gut(G)=\sum_{\{ u,v \} \subseteq V(G)} d_G(u,v)\deg_G (u)\deg_G(v).$$
For more results in this subject or related subjects see 
\cite{Ashrafi}  and \cite{MATCH-Indices}.      

For a graph $G=(V,E)$, the {\it Mycielskian} of G is the graph $\mu(G)$ (or simply, $\mu$) with the disjoint union $V\cup X\cup \{x\}$
as its vertex set and $E\cup \{v_ix_j:~v_iv_j\in E\}\cup \{xx_j:~1\leq j\leq n\}$ as its edge set, 
where $V=\{v_1,v_2,...,v_n\}$ and  $X=\{x_1,x_2,...,x_n\}$, see \cite{Mycielski}.
The Mycielskian and generalized Mycielskians have fascinated graph theorists a great deal. This has resulted
in studying several graph parameters of these graphs like Wiener index, domination number and Zagreb coindices (see \cite{Mycielski-Wiener}, \cite{Fisher} and \cite{Mycielski-Ashrafi}, respectively).
In this paper we study the Gutman index of the Mycielskian graphs. 

\section{Gutman index of the Mycielskian}

In order to determine the Gutman index of Mycielskian graphs, we need the following observations. 
From now on we will always assume that $G$ is a  connected graph, 
$$V(G)=\{v_1,v_2,...,v_n\},~X=\{ x_1,x_2,...,x_n\},~  V(G)\cap X=\emptyset,~ x\notin V(G)\cup X,$$
and $\mu$ is the Mycielskian of $G$, where
$$V(\mu)=V(G)\cup X\cup\{x\}, ~E(\mu)=E(G)\cup \{v_ix_j:~v_iv_j\in E(G)\}\cup\{xx_i:~1\leq i \leq n\}.$$
\begin{observation}  \label{MycielskiDegree}
 Let $\mu$ be the Mycielskian of $G$. Then for each $v\in V(\mu)$ we have
 \begin{eqnarray*}
 \deg_\mu (v)=
 \begin{cases}
 n & v=x  \\
 1+\deg_G(v_i) & v=x_i    \\
 2\deg_G(v_i) & v=v_i.
 \end{cases}
 \end{eqnarray*}
 \end{observation}

\begin{observation}  \label{MycielskiDistance}
In  the Mycielskian $\mu$ of $G$, the distance between two vertices $u,v\in V(\mu)$  are given as follows.
\begin{eqnarray*}
d_\mu(u,v)=
\left\{   \begin{array}{ll}
1 & u=x,~v=x_i \\
2 & u=x,~v=v_i \\
2 & u=x_i,~v=x_j \\
d_G(v_i,v_j) &  u=v_i,~v=v_j,~d_G(v_i,v_j)\leq 3 \\
4 & u=v_i,~v=v_j,~d_G(v_i,v_j)\geq 4 \\
2 & u=v_i,~v=x_j,~i=j \\
d_G(v_i,v_j) & u=v_i,~v=x_j,~i\neq j,~d_G(v_i,v_j)\leq 2 \\
3 & u=v_i,~v=x_j,~i\neq j,~d_G(v_i,v_j)\geq 3.
\end{array} \right.
\end{eqnarray*}
Specially, the diameter of the Mycielskian graph is at most four.
\end{observation}
 
There are $|E(G)|$ unordered pairs of vertices in $V=V(G)$ whose distance is 1, and
$$\sum_{ \substack{ (u,v) \in V\times V \\  d_G(u,v)=1 }}  (\deg_G(u)+\deg_G(v))=
2\sum_{uv\in E(G) } (\deg_G(u)+\deg_G(v))=2M_1(G).$$
\begin{lem}   \label{SumDegreeSums}
Let $G$ be a graph of size $m$ whose vertex set is $V=\{v_1,v_2,...,v_n\}$. Then,
$$\sum_{\{v_i,v_j\}\subseteq V} (\deg_G(v_i)+\deg_G(v_j))=(n-1)2m.$$
\end{lem}
\begin{proof}{
For each $i\in [n]=\{1,2,...,n\}$, $\big{|}\{ \{i,j\}\subseteq [n]:~j\neq i\}\big{|}=n-1$. Therefore,
$$\sum_{\{i,j\}\subseteq [n]} (\deg_G(v_i)+\deg_G(v_j)) = \sum_{i=1}^n (n-1) \deg_G(v_i)=(n-1)2m.$$
}\end{proof}

\begin{lem}   \label{SumDegreeProducts}
Let $G$ be a graph of size $m$ whose vertex set is $V=\{v_1,v_2,...,v_n\}$. Then,
$$\sum_{\{v_i,v_j\}\subseteq V} \deg_G(v_i)\deg_G(v_j)=2m^2-{1\over 2}M_1(G).$$
\end{lem}
\begin{proof}{
The sum of all vertex degrees equals twice the number of edges, hence
\begin{eqnarray*}
(2m)^2=(\sum_{i=1}^n\deg_G(v_i))^2 &=& 
\sum_{i=1}^n (\deg_G(v_i))^2 +2 \sum_{\{v_i,v_j\}\subseteq V} \deg_G(v_i)\deg_G(v_j)  \\
&=& M_1(G)+2 \sum_{\{v_i,v_j\}\subseteq V} \deg_G(v_i)\deg_G(v_j),
\end{eqnarray*}
which completes the proof.
}\end{proof}
 It is a well known fact that almost all graphs have diameter two. This means that graphs of diameter two play an important role in the theory of graphs and their applications.

\begin{theorem}  \label{GutIndex}
Let $G$ be an $n$-vertex graph of size $m$ whose diameter is 2.   If  $\mu$  is the Mycielskian of $G$, then the Gutman index of $\mu$ is given by
$$Gut(\mu)=8 ~\! Gut(G) + 3~\! M_1(G)  +2~\! DD(G) + 4m(m+1)+n(2n-1)+14mn.$$
\end{theorem}

\begin{proof}{
By the definition of Gutman index, we have
\begin{eqnarray*}
Gut(\mu)=\sum_{\{ u,v \} \subseteq V(\mu)} d_\mu(u,v)\deg_\mu (u)\deg_\mu(v).
\end{eqnarray*}
Regarding to the different possible cases which $u$ and $v$ can be choosen from the set $V(\mu)$, the following cases are  considered.
In what follows, the notations are as before and two observations  \ref{MycielskiDegree} and   \ref{MycielskiDistance} are applied for computing degrees and distances in  $\mu$. 
\\
{\bf Case 1.} $u=x$ and $v\in X$:
\begin{eqnarray*}
\sum_{i=1}^n d_\mu(x,x_i)\deg_\mu(x)\deg_\mu(x_i)  = \sum_{i=1}^n n(1+\deg_G(v_i))  =  n^2+2mn.
\end{eqnarray*}
{\bf Case 2.} $u=x$ and $v\in V(G)$:
\begin{eqnarray*}
\sum_{i=1}^n  d_\mu(x,v_i) \deg_\mu(x)\deg_\mu(v_i) = \sum_{i=1}^n 2n(2\deg_G(v_i)) = 8nm.
\end{eqnarray*}
{\bf Case 3.} $\{u,v\}\subseteq X$:
\\
Using Lemma \ref{SumDegreeSums} and Lemma \ref{SumDegreeProducts},  we have
\begin{eqnarray*}
\sum_{\{x_i,x_j\}\subseteq X} \!\!\!\! d_\mu(x_i,x_j)\deg_\mu(x_i)\deg_\mu(x_j)\!\!\! &=& \sum_{\{x_i,x_j\}\subseteq X} 2 (1+\deg_G(v_i))(1+\deg_G(v_j)) \\
&=& \sum_{\{i,j\}\subseteq [n]} 2 (1+\deg_G(v_i))(1+\deg_G(v_j)) \\
&=& \!\!2\!\!\!\sum_{\{i,j\}\subseteq [n]}\!\!\! \big{(}1 + (\deg_G(v_i)+\deg_G(v_j)) + \deg_G(v_i)\deg_G(v_j)\big{)}  \\
&=& 2 \left( \binom{n}{2} +  (n-1)2m +2m^2-{1\over 2}M_1(G) \right)  \\
&=& n(n-1)+4(n-1)m+  4m^2-M_1(G).
\end{eqnarray*}
{\bf Case 4.} $\{u,v\}\subseteq V(G)$:
\\
Since the diameter of $G$ is two,  Observation \ref{MycielskiDistance} implies that $d_\mu (v_i,v_j)=d_G(v_i,v_j)$. 
Hence,
\begin{eqnarray*}
\sum_{\{v_i,v_j\}\subseteq V(G)} d_\mu(v_i,v_j)\deg_\mu(v_i)\deg_\mu(v_j) 
&=&  \sum_{\{v_i,v_j\}\subseteq V(G)} d_G (v_i,v_j) (2\deg_G(v_i))(2\deg_G(v_j)) \\
&=& 4~ Gut(G).
\end{eqnarray*}
{\bf Case 5.} $u=v_i$ and $v=x_i$, $1\leq i\leq n$:
\begin{eqnarray*}
\sum_{i=1}^n d_\mu (v_i,x_i) \deg_\mu(v_i)\deg_\mu (x_i) &=& \sum_{i=1}^n 2 ( 2\deg_G(v_i))(1+\deg_G(v_i) ) \\
&=& 4 \sum_{i=1}^n\left(\deg_G(v_i)+(\deg_G(v_i))^2\right) \\
&=& 4(2m+M_1(G)).
\end{eqnarray*}
{\bf Case 6.} $u=v_i$ and $v=x_j$, $i\neq j$:
\begin{eqnarray*}
\sum_{\substack{\{v_i,x_j\} \subseteq V(\mu)  \\ i\neq j}} \!\!\!\!\!\!\!\!  d_\mu(v_i,x_j)\deg_\mu(v_i)\deg_\mu(x_j) \!\!\!\!
&=& \!\!\!\!\! \sum_{\substack{ \{v_i,x_j\} \subseteq V(\mu) \\ i\neq j}} d_\mu(v_i,x_j) (2\deg_G(v_i))(1+\deg_G(v_j) )  \\
&=&\!\!\! 2\!\!\!\!\!\!\!\! \sum_{\substack{ \{v_i,x_j\} \subseteq V(\mu) \\ i\neq j}} \!\!\!\!\!\!\! d_\mu(v_i,x_j)\deg_G(v_i) + d_\mu(v_i,x_j) \deg_G(v_i)\deg_G(v_j).
\end{eqnarray*}

Since $d_\mu(v_i,x_j)=d_\mu(v_j,x_i)$, $d_G(v_i,v_i)=0$, and using Observation  \ref{MycielskiDistance}, we see that
\begin{eqnarray*}
\sum_{\substack{ \{v_i,x_j\} \subseteq V(\mu) \\ i\neq j}} d_\mu(v_i,x_j) \deg_G(v_i) &=& 
\sum_{\substack{ \{v_i,x_j\} \subseteq V(\mu) \\ i\neq j}} d_G(v_i,v_j) \deg_G(v_i)  \\
 &=&\sum_{\substack{ \{v_i,x_j\} \subseteq V(\mu) }} d_G(v_i,v_j) \deg_G(v_i)  \\
 &=&\sum_{\{i,j\}\subseteq [n]} d_G(v_i,v_j) (\deg_G(v_i) +\deg_G(v_j)) \\
 &=& DD(G).
\end{eqnarray*}

Using similar arguments, it is straightforward  to see that
\begin{eqnarray*}
\sum_{\substack{ \{v_i,x_j\} \subseteq V(\mu) \\ i\neq j } } d_\mu(v_i,x_j)\deg_G(v_i)\deg_G(v_j)&=&
2\sum_{\{i,j\}\subseteq [n]}  d_\mu(v_i,x_j)\deg_G(v_i)\deg_G(v_j)  \\&=&
2\sum_{\{i,j\}\subseteq [n]}  d_G(v_i,v_j)\deg_G(v_i)\deg_G(v_j)  \\
&=&2~Gut(G).    
\end{eqnarray*}

Now the result follows through these six cases.
}\end{proof}

\begin{corollary}  \label{Gmpc}
For the complete bipartite graph $K_{n_1,n_2}$ (which has $n=n_1+n_2$ vertices and $m=n_1n_2$ edges) we have $Gut(\mu(K_{n_1,n_2}))=28m^2+2n^2+15mn-4m-n$, specially for each $n$-vertex star graph $S_n$ we have $Gut(\mu(S_n))=45n^2-76n+32$.
\end{corollary}


\section{Gutman index of  the complement of Mycielskian}

In order to determine the Gutman index of the complement of Mycielskian graphs, we need two following observations. 
\begin{observation}  \label{complementofMycielskiDegree}
 Let $\overline{\mu}$ be the complement of Mycielskian $\mu$ of $G$. Then, for each $v\in V(\overline{\mu})$ 
 \begin{eqnarray*}
 \deg_{\overline{\mu}} (v)=
 \begin{cases}
 n & v=x  \\
 2n-(1+\deg_G(v_i)) & v=x_i    \\
 2n-2\deg_G(v_i) & v=v_i.
 \end{cases}
 \end{eqnarray*}
 \end{observation}

\begin{observation}  \label{DistanceMycielskiComp}
In  the complement of Mycielskian $\mu$ of $G$, the distance between two vertices $u,v\in V(\overline{\mu})$  are given as follows.
\begin{eqnarray*}
d_{\overline{\mu}}(u,v)=
\left\{   \begin{array}{ll}
2& u=x,~v=x_i \\
1 & u=x,~v=v_i \\
1 & u=x_i,~v=x_j \\
1 &  u=v_i,~v=v_j,~d_G(v_i,v_j)>1 \\
2 & u=v_i,~v=v_j,~d_G(v_i,v_j)= 1 \\
1 & u=v_i,~v=x_j,~i=j \\
1 & u=v_i,~v=x_j,~i\neq j,~d_G(v_i,v_j)>1\\
2 & u=v_i,~v=x_j,~i\neq j,~d_G(v_i,v_j)=1.
\end{array} \right.
\end{eqnarray*}
Specially, the diameter of $\overline{\mu}$ is exactly 2.
\end{observation}


\begin{theorem}  \label{Gut complement}
Let $G$ be an $n$-vertex graph of size $m$ and let  $\overline{\mu}$  be the complement of the Mycielskian $\mu$ of $G$.  Then, the Gutman index of $\overline{\mu}$ is given by
$$Gut(\overline{\mu})=
8M_2(G)-(10n+{1\over 2})M_1(G)+  n \bigg{(}  2n^2(4n-1)+{n-1\over 2} \bigg{)}  +6mn(1-2n)+2m(9m-1).
$$
\end{theorem}
\begin{proof}{
By the definition, we have 
\begin{eqnarray*}
Gut(\overline{\mu})=\sum_{\{ u,v \} \subseteq V(\overline{\mu})} d_{\overline{\mu}}(u,v)\deg_{\overline{\mu}}(u)\deg_{\overline{\mu}}(v).
\end{eqnarray*}
We  consider the following cases.
For computing degrees and distances in $\overline{\mu}$ two observations  \ref{complementofMycielskiDegree} and   \ref{DistanceMycielskiComp} are applied.
\\
{\bf Case 1.} $u=x$ and $v\in X$:
\begin{eqnarray*}
\sum_{i=1}^n  d_{\overline{\mu}}(x,x_i)\deg_{\overline{\mu}}(x)\deg_{\overline{\mu}}(x_i) &=& \sum_{i=1}^n 2n(2n-1-\deg_G(v_i)) \\
&=& 2n((2n-1)n-2m).
\end{eqnarray*}
{\bf Case 2.} $u=x$ and $v\in V(G)$:
\begin{eqnarray*}
\sum_{i=1}^n  d_{\overline{\mu}}(x,v_i)\deg_{\overline{\mu}}(x)\deg_{\overline{\mu}}(v_i)&=& \sum_{i=1}^n n(2n-2\deg_G(v_i)) \\
&=& 2n~(n^2-2m).
\end{eqnarray*}
{\bf Case 3.} $\{u,v\}\subseteq X$:

Using Lemma \ref{SumDegreeSums}  and Lemma \ref{SumDegreeProducts}, we see that
\begin{eqnarray*}
\sum_{\{x_i,x_j\}\subseteq X} \!\! d_{\overline{\mu}}(x_i,x_j)\deg_{\overline{\mu}}(x_i)\deg_{\overline{\mu}}(x_j) 
\!\!\! &=& \sum_{\{x_i,x_j\}\subseteq X} (2n-1-\deg_G(v_i))~(2n-1-\deg_G(v_j)) \\
&=& \sum_{\{i,j\}\subseteq [n]} (2n-1-\deg_G(v_i))~(2n-1-\deg_G(v_j)) \\
&=& \sum_{\{i,j\}\subseteq [n]} \bigg{(}(2n-1)^2-(2n-1)(\deg_G(v_i)+\deg_G(v_j)) \\ 
&&  + \deg_G(v_i)\deg_G(v_j)\bigg{)}  \\
&=& \!\!\! \binom{n}{2} (2n-1)^2-(2n-1)(n-1)2m+2m^2-{1 \over 2}M_1(G).
\end{eqnarray*}
{\bf Case 4.} $\{u,v\}\subseteq V(G)$:

By Observation \ref{DistanceMycielskiComp}, $d_{\overline{\mu}}(v_i,v_j)$ is 1 whenever $v_iv_j\notin E(G)$ and is 2 otherwise.  Also,
$$\big{\{} \{v_i,v_j\}\subseteq V:~i\neq j,~v_iv_j\notin E(G)\big{\}}=\big{\{} \{v_i,v_j\}\subseteq V:~i\neq j\big{\}} \setminus \big{\{} \{v_i,v_j\}\subseteq V:~v_iv_j\in E(G)\big{\}}$$
where, $V=V(G)$. Thus,
\begin{eqnarray*}
\sum_{\{v_i,v_j\}\subseteq V(G)} d_{\overline{\mu}}(v_i,v_j)\deg_{\overline{\mu}}(v_i)\deg_{\overline{\mu}}(v_j)
\!\!\! &=& \sum_{v_iv_j\notin E(G)}1 (2n-2\deg_G(v_i))(2n-2\deg_G(v_j)) \\
 && + \sum_{v_iv_j\in E(G)}2(2n-2\deg_G(v_i))(2n-2\deg_G(v_j)) \\
 &=& \sum_{\{v_i,v_j\}\subseteq V(G)} (2n-2\deg_G(v_i))(2n-2\deg_G(v_j)) \\
 && + \sum_{v_iv_j\in E(G)} (2n-2\deg_G(v_i))~(2n-2\deg_G(v_j)) 
\end{eqnarray*}
Now, two lemmas \ref{SumDegreeSums}  and  \ref{SumDegreeProducts} imply that
\begin{eqnarray*}
\sum_{\{v_i,v_j\}\subseteq V(G)} \!\! (2n-2\deg_G(v_i))(2n-2\deg_G(v_j))  =
4n^2 \binom{n}{2} -4n(n-1)2m +4(2m^2-{1\over 2}M_1(G)).
\end{eqnarray*}
Also, the definitions of first and second Zagreb indices imply that
\begin{eqnarray*}
\sum_{v_iv_j\in E(G)} (2n-2\deg_G(v_i))(2n-2\deg_G(v_j)) =
 4n^2m-4n~M_1(G)+4~M_2(G).
\end{eqnarray*}
{\bf Case 5.} $u=v_i$ and $v=x_i$, $1\leq i\leq n$:
\begin{eqnarray*}
\sum_{i=1}^n d_{\overline{\mu}} (v_i,x_i) \deg_{\overline{\mu}}(v_i)\deg_{\overline{\mu}} (x_i) 
&=& \sum_{i=1}^n ( 2n-2\deg_G(v_i))(2n-1-\deg_G(v_i)) \\
&=& 2n^2(2n-1)-(6n-2)2m+2M_1(G).
\end{eqnarray*}
{\bf Case 6.} $u=v_i$ and $v=x_j$, $i\neq j$:
\\
By Observation \ref{DistanceMycielskiComp}, $d_{\overline{\mu}}(v_i,x_j)=d_{\overline{\mu}}(v_j,x_i)$  is 1 when $v_iv_j\notin E(G)$, otherwise is 2.  Also,
$$\big{\{} (v_i,v_j):~i\neq j,~v_iv_j\notin E(G)\big{\}} = \big{\{} (v_i,v_j):~i\neq j\big{\}} \setminus \big{\{} (v_i,v_j):~v_iv_j\in E(G)\big{\}}.$$
Thus,
\begin{eqnarray*}
\sum_{\substack{\{v_i,x_j\} \subseteq V(\overline{\mu})  \\ i\neq j}} \!\!\! d_{\overline{\mu}}(v_i,x_j)\deg_{\overline{\mu}}(v_i)\deg_{\overline{\mu}}(x_j)
\!\!\! &=& \sum_{\substack{(v_i,v_j) \\ v_iv_j\notin E(G)}}  1 (2n-2\deg_G(v_i))(2n-1-\deg_G(v_j)) \\
 && +\!\!\! \sum_{\substack{(v_i,v_j) \\ v_iv_j\in E(G)}}  2 (2n-2\deg_G(v_i))(2n-1-\deg_G(v_j)) \\
 &=& \sum_{\substack{(v_i,v_j)  \\ i\neq j}}  (2n-2\deg_G(v_i))(2n-1-\deg_G(v_j)) \\
 && + \!\! \sum_{\substack{(v_i,v_j) \\ v_iv_j\in E(G)}}  (2n-2\deg_G(v_i))(2n-1-\deg_G(v_j)) \\
\end{eqnarray*}
Each vertex $v_j$ can be paired with $n-1$ vertices $v_i$ as $(v_i,v_j)$, $i\neq j$.  
Hence
$\sum_{_{\substack{\\ (v_i,v_j)}}}\!\! \deg_G(v_j)\!= (n-1)\sum_{j=1}^n \deg_G(v_j)$ which is equal to $(n-1)2m$.
Also, note that $\sum_{_{\substack{\\ (v_i,v_j)}}} \deg_G(v_i)\deg_G(v_j)$ equals $2 \sum_{_{\{v_i,v_j\}}} \deg_G(v_i)\deg_G(v_j)$.
Now, since $\big{|}\{(v_i,v_j):~i\neq j\}\big{|}=n(n-1)$, we obtain
\begin{eqnarray*}
\sum_{\substack{(v_i,v_j)  \\ i\neq j}} (2n-2\deg_G(v_i))(2n-1-\deg_G(v_j)) \!\!\! &=& 
 2n(2n-1)n(n-1)  - 2n(n-1)2m \\
&& - 2(2n-1)(n-1)2m+4(2m^2-{1 \over 2}M_1(G)). 
\end{eqnarray*}
Note that $\big{|} \{ (v_i,v_j):~v_iv_j\in E(G) \} \big{|}=2m$ and $\sum_{_{\substack{(v_i,v_j) \\ v_iv_j\in E(G) }}} \deg_G(v_i)=\sum_{i=1}^n (\deg_G(v_i))^2$, because each vertex $v_i$ has $\deg_G(v_i)$ neighbours and appears  $\deg_G(v_i)$  times in the desired summation.
Thus, using Lemma \ref{SumDegreeProducts}, we see that
\begin{eqnarray*}
\sum_{\substack{(v_i,v_j)   \\ v_iv_j\in E(G)}} (2n-2\deg_G(v_i))(2n-1-\deg_G(v_j)) \!\!\! &=&
 2n(2n-1)2m - 2nM_1(G) \\
&& - 2(2n-1)M_1(G) + 4 M_2(G).
\end{eqnarray*}

Now the result follows through the cases 1 to 6.
}\end{proof}
Let  $P_n$ and  $C_n$ denote the path and the cycle on $n$ vertices, respectively, where $n \geq 3$. 
It is well known (see \cite{Ashrafi}) that
$M_1(P_n)= 4n-6$, $M_2(P_n)=4(n-2)$, $M_1(C_n)= 4n=M_2(C_n)$.
\begin{corollary}  \label{Gmpc}
The following statements hold.
\begin{itemize}
\item[i)] $Gut(\overline{\mu(P_n)})= 2n^3(4n-7)-{7\over 2}n(n-13)-41 $ for each $n\geq 3$.
\item[ii)]  $Gut(\overline{\mu(C_n)})= 2n^3(4n-7)-{1\over 2}n(31n-55) $ for each $n\geq 3$.
\item[iii)] $Gut(\overline{\mu(K_{n_1,n_2})})= 2m(13m-1)+ n \bigg{(}  2n^2(4n-1)+{n-1\over 2} \bigg{)} -mn(22n-{11\over2}) $ in which $n=n_1+n_2$ and  $m=n_1n_2$.
\item[iV)]  $Gut(\overline{\mu(S_n)})= 2n^3(4n-1)-2(n-1)(11n^2-16n+14) $ for each $n\geq 2$.
\item[V)]  $Gut(\overline{\mu(K_n)})= {1\over 2}n^3(n+11)+4n(n-1)$ for each $n\geq 2$.
\end{itemize}
\end{corollary}

For example we have $Gut(\overline{\mu(P_3)})=Gut(\overline{\mu(S_3)})=Gut(\overline{\mu(K_{1,2})})=334$ and $Gut(\overline{\mu(C_3)})=Gut(\overline{\mu(K_3)})=213$.


\begin{thebibliography}{99}





\bibitem{JAMC1}
Basavanagoud, B., Patil,  S., A note on hyper-Zagreb coindex of graph operations, {\it J. Appl. Math. Comput.}, {\bf 53} No.1, 647-655 (2017).

\bibitem{Dobrynin}
Dobrynin, A. A., Kochetova,  A. A. , Degree Distance of a Graph: A Degree Analogue of the Wiener Index, {\it J.
Chem. Inf. Comput. Sci.}, {\bf 34},  1082-1086 (1994).

\bibitem{Mycielski-Wiener}
Eliasi, M., Raeisi,  G., Taeri, B., Wiener index of some graph operations, {\it Discret. Appl. Math.}, {\bf 160},  1333-1344 (2012).

\bibitem{Fisher}
Fisher, D.C.,   McKena, P.A.,   Boyer, E.D.,   Hamiltonicity, diameter, domination, packing and biclique
partitions of Mycielski’s graphs, {\it Discret. Appl. Math.}, {\bf 84},  93-105 (1998).

\bibitem{Schultz-GutmanIndex}
 Gutman, I., Selected Properties of the Schultz Molecular Topological Index, {\it J. Chem. Inf. Comput. Sci.}, {\bf 34},   1087-1089 (1994).


\bibitem{GutmanTrinajstic}
 Gutman,  I.,  Trinajsti$\acute{\mbox{c}}$, N., Graph theory and molecular orbitals. Total $\pi$-electron energy of
alternant hydrocarbons, {\it Chem. Phys. Lett.}, {\bf 17},  535-538 (1972).


\bibitem{Mycielski-Ashrafi}
Hua, H.,   Ashrafi, A. R.,  Zhang, L., More on Zagreb coindices of graphs, {\it Filomat}, {\bf 26},  1215-1225 (2012).



\bibitem{Ashrafi}
 Khalifeh, M. H.,  Yousefi-Azari, H.,  Ashrafi, A. R.,  Wagner, S., Some new results on
distance-based graph invariants, {\it  Eur. J. Comb.}, {\bf 30},  1149-1163 (2009).

\bibitem{MATCH-Degreedistance}
 Ili$\acute{\mbox{c}}$, A.,  Klav$\check{\mbox{z}}$ar, S.,  Stevanovi$\acute{\mbox{c}}$, D., Calculating the degree distance of partial Hamming
graphs, {\it MATCH Commun. Math. Comput. Chem.}, {\bf 63},   411-424 (2010).

 
\bibitem{Mycielski}
 Mycielski, J.,    Sur le colouriage des graphes, {\it Colloq. Math.}, {\bf 3},   161-162 (1955).
 
 \bibitem{JAMC2}
  Tan, S.W.,  Lin, Y., The largest Wiener index of unicyclic graphs given girth or maximum degree, {\it J. Appl. Math. Comput.},  {\bf 53},  No.1, 343-363 (2017).



\bibitem{MATCH-Indices} 
  Xu, K.,  Liu, M.,  Das, K. C.,  Gutman, I., Furtula, B., A survey on graphs extremal with respect to distance-based
topological indices, {\it MATCH Commun. Math. Comput. Chem.}, {\bf 71},  461-508 (2014).


\end{thebibliography}
\end{document}